\documentclass{article}
\PassOptionsToPackage{gray}{xcolor}
\usepackage[monochrome]{xcolor}
\usepackage{notes}
\usepackage{url}

\let\setminus\smallsetminus
\usepackage{diagbox}

\makeop{HS}
\title{Bigness of the tangent bundle of del Pezzo surfaces and $D$-simplicity}
\author{Devlin Mallory\thanks{The author was supported by NSF Graduate Research Fellowship grant DGE-1256260, as well as partially supported by NSF grant DMS-1701622.}}
\makeop{sing}

\DeclareMathOperator{\Diff}{\mathcal{D}iff}

\makeop{Cont}

\makeop{mld}
\makeop{Tan}
\makeop{lct}

\theoremstyle{theorem}
\newtheorem{thm}{Theorem}
\newtheorem*{thmm}{Theorem}
\numberwithin{thm}{section}
\newtheorem{lem}[thm]{Lemma}
\newtheorem{cor}[thm]{Corollary}

\newtheorem{quest}[thm]{Question}
\newtheorem{prop}[thm]{Proposition}
\theoremstyle{definition}
\newtheorem{dfn}[thm]{Definition}
\newtheorem{exa}[thm]{Example}
\newtheorem{rem}[thm]{Remark}

\begin{document}
\maketitle
\abstract{We consider the question of simplicity of a ring $R$ under the action of  its ring of differential operators $D_R$. We give examples to show that even when $R$ is Gorenstein and has rational singularities $R$ need not be a simple $D_R$-module; for example, this is the case when $R$ is the homogeneous coordinate ring of a smooth cubic surface. Our examples are homogeneous coordinate rings of smooth Fano varieties, and our proof proceeds by showing that the tangent bundle of such a variety need not be big. We also give a partial converse showing that when $R$ is the homogeneous coordinate ring of a smooth projective variety $X$, embedded by some multiple of its canonical divisor, then simplicity of $R$ as a $D_R$-module implies that $X$ is Fano and thus $R$ has rational singularities.} 

\section{Introduction}

Given a $k$-algebra $R$, one can consider the ring $D_{R/k}$ of $k$-linear differential operators on $R$. When $R=k[x_1,\dots,x_n]$ (or when $R$ is a smooth $k$-algebra), $D_{R/k}$ is well-studied and has several nice properties; however, when $R$ is not a smooth $k$-algebra,  $D_{R/k}$ can be quite mysterious. For example, \cite{BGG} showed that if $R=\C[x,y,z]/(x^3+y^3+z^3)$, then $D_{R/\C}$ is not finitely generated over $\C$, not left- or right-Noetherian, and that $R$ is not a simple $D_{R/\C}$-module.

We consider the following questions:
\begin{enumerate}
\item If $\Spec R$ has rational singularities, when is $D_{R/k}$ simple? (\cite[Question~0.13.1]{LevStaf})
\item When is $R$ a simple $D_{R/k}$-module? (\cite[Question~0.13.3]{LevStaf})
\end{enumerate}

\begin{rem}
As we will discuss in Remark~\ref{implic}, if $D_{R/k}$ is simple then $R$ is a simple $D_{R/k}$-module, and if $R$ is a graded ring that is a simple $D_{R/k}$-module, then $D_{R/k}$ must contain differential operators of negative degree.
\end{rem}

In the setting of finite-type $\C$-algebras,
\cite[Question~5.1]{bigness} asked whether in (2) above it's sufficient for $R$ to have Gorenstein rational singularities.
This criteria was motivated in part by \cite[Theorem~2.2]{smith}, which showed that in characteristic $p$ an $F$-pure ring $R$ is a simple $D_R$-module if and only if $R$ is strongly $F$-regular; thus, one might expect a ``mildly'' singular ring $R$ in characteristic 0 to be a simple $D_R$-module. 

In this note, we give a negative answer to Hsiao's question, illustrating the differing behavior of differential operators in characteristic $p$ and characteristic $0$.

\begin{thm}
\label{notsuff}
There are  Gorenstein (graded) $\C$-algebras $R$ with rational singularities such that $D_{R/k}$ contains no differential operators of negative degree, and thus such that $R$ is not a simple $D_{R/k}$-module and $D_{R/k}$ is not simple.
One example is 
$R=\C[x,y,z,w]/(x^3+y^3+z^3+w^3)$.
\end{thm}

We also give a positive answer to the necessity of these conditions in the special case where $R$ is the homogeneous coordinate ring of a smooth projective variety embedded by a multiple of its canonical divisor:

\begin{thmm}[Theorem~\ref{partnec}]
Let $R$ be a 
normal $\Q$-Gorenstein graded $\C$-algebra, generated in degree 1, with an isolated singularity
If $R$ is a simple $D_R$-module, or merely admits a differential operator of negative degree,  then $R$ has klt singularities, and thus rational singularities.
\end{thmm}

Both main theorems are proved by working with the smooth complex variety $X=\Proj R$, and using the following observation of \cite{bigness}:

\begin{thm}[{{\cite[Theorem~1.2]{bigness}}}]
Let  $X$ be a smooth complex projective variety and $L$ an ample line bundle. Set $R=S(X,L)=\bigoplus_m H^0(X,L^m)$.
If $R$ is a simple $D_R$-module then $T_X$ is~big.
\end{thm}

In Section~\ref{bigvec} we discuss the notion of bigness of vector bundles; in the context of this theorem, bigness of $T_X$ is equivalent to,
 for any $e>0$,
 the existence of a nonzero global section of $H^0(\Sym^m T_X\otimes L^{-e})$ for some $m\gg0$.

The singularities of rings of the form
$R=S(X,L)=\bigoplus_m H^0(X,L^m)$ (i.e., rings that are section rings of polarized smooth complex projective varieties) translate to properties of the embedding $X\hookrightarrow \P^N$ determined by $H^0(X,L^m)$ for some $m$ large enough:
\begin{enumerate}
\item $R$ is always normal.
\item $R$ is Gorenstein if and only if $L =\O_X(aK_X)$ for some $a\in \Z$, and $\Q$-Gorenstein if and only if $L^{\otimes b} =\O_X(aK _X)$ for $a,b\in \Z$.
\item $R$ is klt if and only if it's $\Q$-Gorenstein and $a<0$, i.e., if and only if it's $\Q$-Gorenstein and $-K_X$ is ample.
\end{enumerate}

If $R$ is klt then it has rational singularities, and the converse is true if $R$ is Gorenstein. Thus, to give a counterexample to the sufficiency of Gorenstein rational singularities for $D_R$-simplicity, we find a variety $X$ with $-K_X$ ample (i.e., $X$ is Fano) and $T_X$ not big. 
Thus, Theorem~\ref{notsuff} will follow from

\begin{thm}[Theorem~\ref{main}]
Let $X$ be a del Pezzo surface of degree 3, i.e., a smooth cubic surface.
$T_X$ is not big; in fact, $H^0(X,\Sym^m T_X)=0$ for all $m$.
\end{thm}

\begin{rem}
After the first version of this paper was made public, the author realized that Theorem~\ref{main} also follows from \cite[Theorem~B]{BO}. The proof we give is distinct, and more direct but less general; see Section~\ref{new} for a discussion of their results.
\end{rem}

Similarly, to address the necessity if $R=S(X,L)$ is the homogeneous coordinate ring of a smooth projective variety embedded by a power of its canonical divisor (i.e., such that either $K_X$ or $-K_X$ is ample), we show that $T_X$ big implies that $-K_X$ ample.
In fact, we note the following statement (likely known to experts), which immediately implies our statement on necessity:

\begin{prop}
Let $X$ be a smooth complex projective variety.
If $T_X$ is big, then $X$ is uniruled.
\end{prop}

We begin by discussing differential operators and $D$-simplicity in Section~\ref{diffops}.
We then recall the definition and properties of big vector bundles in Section~\ref{bigvec}, and discuss the connection between $D$-simplicity and bigness of the tangent bundle in Section~\ref{tangentbundle}.
In Section~\ref{delpezzos}, we show that a Fano variety  need not have big tangent bundle, by examining the tangent bundles of some del Pezzo surfaces. We show that if $X$ is a del Pezzo surface of degree $\leq 4$, then $T_X$ is not big (and thus in particular the tangent bundle of a smooth cubic surface is not big).
Section~\ref{new} discusses how the results discussed in Section~\ref{delpezzos} also follows from work of \cite{BO,OL}.
In Section~\ref{conv}, we prove Theorem~\ref{partnec} by showing that, if $R=S(X,L)$ for some smooth complex projective $X$ with $L$ very ample and a multiple of $K_X$, then $D_R$-simplicity of $R$ forces $X$ to be Fano and thus $R$ to have klt (thus rational) singularities.
Sections~\ref{poschar} and \ref{zerochar} contain no new results, but
compare and contrast our results with known results in positive characteristic and characteristic 0 respectiveley.

\begin{rem}
We note that after the initial draft of this note was made public, there has been additional progress via work of \cite{HLS}: they completely classify which del Pezzo surfaces have big (or pseudoeffective) tangent bundle via different methods, and in particular recover the two examples we treat here (the del Pezzos of degrees 3 and 4). In addition, they are able to  treat also the case of a hypersurface in $\P^n$, as well as certain del Pezzo threefolds.
For further discussion, see Remark~\ref{otherpaper}.
\end{rem}
\subsection*{Acknowledgements}
I'm grateful to Anurag Singh for introducing us to this problem, and for helpful conversations.
This work also benefited from useful discussions with Jack Jeffries, Monica Lewis, Eamon Quinlan-Gallego,
and Mel Hochster.
I'm grateful as well to Jen-Chieh Hsiao
for helpful correspondence and his paper \cite{bigness} which provided the key tools for this note.
I would also like to thank Feng Shao for pointing out an error in a previous version of the paper.
Finally, I thank my advisor 
Mircea Musta\c{t}\u{a}
for many helpful conversations and suggestions over the course of this project.

\section{Differential operators and singularities}
\label{diffops}

We begin with a brief discussion of differential operators on rings.
Let $A$ be a commutative noetherian ring and
$R$ a commutative $A$-algebra. 
The (noncommutative) ring $D_{R/A}$ of $A$-linear differential operators on $R$  is defined as follows: let $D_{R/A}^0=\Hom_R(R,R)=R$, where for $r\in R$ we write $\overline r $ for the multiplication by $r$ map $R\to R$.  Define inductively
$$
D_{R/A}^i :=\set{\delta \in \Hom_A(R,R): [\overline r,\delta]\in D_{R/A}^{i-1} },
$$
and set
$$
D_{R/A}=\bigcup_i D_{R/A}^i.
$$
We note $D_{R/A}$ is a subring of $ \Hom_{A}(R,R)$ and thus $R$ carries a canonical $D_{R/A}$-module structure.

\begin{exa}
If $k$ is a field and $R=k[x_0,\dots,x_n]$, then 
$$
D_{R/k}=R\Bigl\<\frac{1}{\a_0!\cdots \a_n!}
\Bigl(\frac{\d }{\d x_0}\Bigr)^{\a_0}
\cdots
\Bigl(\frac{\d}{\d x_n}\Bigr)^{\a_n}: \a_i\in \N
 \Bigr\>
$$
is the familiar ring generated by the partial differential operators (the coefficients are irrelevant if $\Char k=0$, but necessary when $\Char k>0$).
\end{exa}

\begin{rem}
When $R=\bigoplus_d R_d$ is a graded ring, one can define homogeneous differential operators of degree $e$: one writes 
$$
(D_{R/k})_e = \set{\delta\in D_{R/k}: \delta(R_d)\subset R_{d+e} \text{ for all } d }.
$$
Then $D_{R/k}$ is itself a graded ring.
\end{rem}

For the rest of the paper, $A$ will be a field $k$, most often $\C$, and $R$ a finitely generated $k$-algebra. We will write simply $D_R$ for $D_{R/k}$.

When $R$ is not a smooth $k$-algebra, the behavior of $D_R$ can be quite subtle, even when $R$ is a ``nice'' ring (e.g., a complete intersection). For example, we have the following example:

\begin{exa}[\cite{BGG}]
Let $R=\C[x,y,z]/(x^3+y^3+z^3)$ be the affine cone over a smooth elliptic curve. Then
$D_R$ has no differential operators of negative degree, $D_R$ is not a finitely generated $\C$-algebra, and is neither left- nor right-Noetherian.
We note here that since $D_R$ has no differential operators of negative degree, the maximal homogeneous ideal $(x,y,z)$ is a proper sub-$D_R$-module of~$R$.
\end{exa}

There are a variety of ways to use properties of $D_R$ to describe the singularities of the ring $R$: one can consider noetherianity of $D_R$, finite generation, generation by derivations, freeness of the $R$-module $D^1_R$, and more (see, for example, \cite{LevStaf,smith,smithvdB,nakaiconj}).
In particular, \cite{LevStaf} posed the following questions:

\begin{enumerate}
\item If $\Spec R$ has rational singularities, when is $D_R$ simple? (\cite[Question~0.13.1]{LevStaf})
\item When is $R$ a simple $D_R$-module? (\cite[Question~0.13.3]{LevStaf})
\end{enumerate}

We recall the definition of rational and klt singularities:
\begin{dfn}
A $\C$-algebra $R$ has \emph{rational singularities} if for some resolution of singularities $f:Y\to \Spec R$ we have $f_*\O_Y=\O_{\Spec R}$ and $R^if_*\O_Y=0$ for $i>0$.
\end{dfn}

\begin{rem}
\label{propsings}
We will also use the notion of \emph{klt singularities} of $R$ (by which we mean of $\Spec R$), which are of importance in the minimal model program. We omit this definition here, 
as we will not make any use of the definition (nor that of rational singularities) in this note. (We do note that klt singularities are by definition $\Q$-Gorenstein.)
We assemble here the only facts we will need:
\begin{itemize}
\item If $R$ has klt singularities then it has rational singularities \cite{Elkik}.
\item If $R$ is Gorenstein and has rational singularities then it has klt singularities \cite[Proposition~3.1]{SchwedeTakagi}.
\item Let $X$ be a smooth projective variety and $L$ a very ample line bundle. Then $S(X,L):=\bigoplus H^0(X,L^{\otimes m})$ is $\Q$-Gorenstein if and only if $L^{\otimes b}\cong \O_X(aK_X)$ for some $a,b\in \Z$, and has rational singularities if and only if $-K_X$ is ample \cite[Lemma~3.1]{kollar}.
\end{itemize}
\end{rem}

The property in (2) above is called $D$-simplicity:

\begin{dfn}
A $k$-algebra
$R$ is called $D$-simple if $R$ is a simple $D_R$-module.
\end{dfn}

For clarity, we will sometimes say $R$ is $D_R$-simple.

\begin{rem}
\label{implic}
We note the following implications:
\begin{itemize}
\item 
If $D_R$ is a simple ring, then $R$ is $D$-simple: to see this, if $R$ is not $D$-simple, let $I$ be a proper nonzero ideal of $R$ such that $D_R I \subset I$;  one can then check that $\set{\delta \in D_R: \delta R \subset I}$ is a nonzero proper ideal of $D_R$. The converse is not true (see, e.g, \cite[0.13]{LevStaf}).
\item 
If $R$ is a graded ring, and $R$ is $D$-simple, then $D_R$ must have differential operators of negative degree, i.e., $(D_R)_e\neq0 $ for some $e<0$.
\end{itemize}
\end{rem}

\begin{rem}
$D$-simplicity is known to impose certain conditions on $R$. If $R$ is $D$-simple then:
\begin{itemize}
\item If $R$ is reduced then $R$ must be a domain.
\item $R$ is Cohen--Macaulay \cite[Theorem~6.2.5]{vdB}.
\item If $k$ is a perfect field of positive characteristic, and $R$ is $F$-pure, then $R$ is $D$-simple if and only if $R$ is strongly $F$-regular \cite[Theorem~2.2]{smith}.
\end{itemize}
The following classes of rings  are known to be $D$-simple:
\begin{itemize}
\item If $T$ is $D_T$-simple, and the inclusion of an $T$-submodule $R\hookrightarrow T$ splits as a map of $R$-modules, then $R$ is $D_R$-simple \cite[Proposition~3.1]{smith}.
Thus, for example, toric rings are $D$-simple.
\item Invariant subrings of polynomial rings under the action of classical algebraic groups \cite{LevStaf}.
\end{itemize}
\end{rem}

The examples listed above, as well as analogies between strong $F$-regularity in characteristic~$p$ and klt singularities in characteristic 0, motivate the following more specific formulation:

\begin{quest}[{{\cite[Question~5.1]{bigness}}}]
\label{mainquest}
If $R$ is a finitely generated Gorenstein $\C$-algebra such that $\Spec R$ has rational singularities, is $R$ then $D$-simple?
\end{quest}

We note that since $R$ is assumed to be Gorenstein it's equivalent to ask whether $\Spec R$ has klt singularities.

\begin{rem}
\cite{LevStaf} gives an example of a ring with rational singularities which is not $D$-simple; the ring in question is obtained as the quotient of $\C[x,y,z]/(x^3+y^3+z^3)$ under a $\Z/3\Z$-action, which is a 2-dimensional normal isolated rational singularity. This example is why one must impose the Gorenstein condition in the phrasing of \cite[Question~5.1]{bigness}.
\end{rem}

Our Theorem~\ref{notsuff}
exhibits
a ``mildly'' singular ring $R$  which not only is not $D$-simple (and thus such that $D_R$ is itself not a  simple ring), but which does not have any differential operators of ngegative degree. This indicates
that one must impose fairly strong conditions on $R$ to obtain a sufficient condition for Question~0.13.1 of \cite{LevStaf} on the simplicity of $D_R$.

\section{Positivity of vector bundles}
\label{bigvec}
The rest of the paper will use various notions of positivity for vector bundles. We recall some definitions and properties here, largely following \cite[Chapter~6]{LazarsfeldII}.

Let $X$ be any variety and $E$ a locally free sheaf of rank $r$ on $X$. We write $\pi:\P E\to X$ for the projective bundle of 1-dimensional quotients of $E$. $\P E$ carries a tautological line bundle $\O_{\P E}(1)$, and we have that $\pi_* \O_{\P E}(m)=\Sym ^m E$ for $m\geq 0$.

\begin{dfn}
 $E$ is said to be ample, nef or big if the line bundle $\O_{\P E}(1)$ is ample, nef, or big respectively.
\end{dfn}

\begin{rem}
There are conflicting conventions for defining bigness of vector bundles. The definiton we take here is elsewhere called ``L-big'' (for ``Lazarsfeld-big''). There is also the stronger notion of ``V-big'' (for ``Viehwig-big'').
These differ even in quite simple cases: for example, $\O_{\P^1}\oplus \O_{\P^1}(1)$ is L-big but not V-big. For a detailed discussion of the different notions of positivity generally and bigness specifically, see \cite{posbund,jabbusch}.
\end{rem}

\begin{rem}
A line bundle $L$ on a variety $X$ of dimension $N$ is big if and only if 
$$\lim_{m\to \infty} \frac{h^0(X,L^{\otimes m})}{m^N}>0.$$ 
We can give a similar characterization for bigness of vector bundles:

Say $E$ is a rank-$r$ vector bundle on a variety $X$ of dimension $n$.
Because $\pi_*\O_{\P E}(1)=\Sym^m E$, we have that
$$
H^0(\P E,\O_{\P E}(m))=H^0(X,\Sym^m E).
$$
Since $\P E$ has dimension $n+r-1$, we have that $E$ is big if and only if
$$
\lim_{m\to \infty} \frac{h^0(X,\Sym^m E) }{m^{n+r-1}} >0.
$$
\end{rem}

The following characterization of bigness will be crucial in Section~\ref{tangentbundle}:

\begin{prop}[\cite{bigness}]
\label{bigsimple}
Let $L$ be an ample line bundle on $X$.
$E$ is big if and only if for all $e>0$ there exists 
$m\geq 0$ such that
$H^0(X,\Sym^m E\otimes L^{-e})\neq 0$.
\end{prop}

We reproduce the proof of \cite{bigness} here for convenience (note that just the ``only if'' implication is stated there, although the converse direction is straightforward).

\begin{proof}
Say $H^0(X,\Sym ^m E\otimes L^{-e})=H^0(\O_{\P E}(m)\otimes \pi^* L^{-e}) \neq 0$, i.e., $\O_{\P E}(m) \otimes \pi^*L^{-e}$  is effective (and thus so are all its positive tensor powers).
Since $\O_{\P E}(1)$ is $\pi$-ample, we have that $\O_{\P E}(1)\otimes \pi^* L^j$ is ample for $j\gg0$; choosing $j=N e$ for $N\gg0$, we have
$$
\O_{\P E}(mN+1)= \underbrace{(\O_{\P E}(1)\otimes \pi^* L^{Ne})}_{\text{ample}}\otimes \underbrace{(\O_{\P E}(mN)\otimes L^{-Ne})}_{\text{effective}}
$$
Thus, we have that $\O_{\P E}(mN+1)$ can be decomposed as the product of an ample line bundle and an effective line bundle, and is thus  big; hence $\O_{\P E}(1)$ is big as well.

Conversely,
if $E$ is big, then Kodaira's lemma (Lemma~\ref{kodaira} below) implies that for any $e$ such that $L^e$ is effective, there exists $m$ such that $\O_{\P E}(m)\otimes \pi^*L^{-e}$ has a section for some $m\gg0$. Then 
$$
0\neq H^0(\P E,\O_{P E}(m)\otimes \pi^*L^{-e})=H^0(X,\Sym^m E \otimes L^{-e}),
$$
concluding the proof.
\end{proof}

\begin{lem}[Kodaira's lemma]
\label{kodaira}
Let $Y$ be a normal variety,
$A$ a big divisor and $D$ an effective divisor.
Then 
$$
H^0(Y,\O_Y(mA-D))\neq 0
$$
for all $m$ sufficiently large and divisible.
\end{lem}

For a proof, see \cite[Proposition~2.2.6]{Lazarsfeld}.

Finally, we note here a fact we will use throughout:

\begin{lem}
\label{sessym}
Let $k$ be a field and
$X$ a $k$-scheme, and
let 
$$
0\to L\to E\to F\to 0
$$
be a short exact sequence of vector bundles, with $L$ a line bundle.
Then for any $m>0$ we have a short exact sequence
$$
0\to L\otimes \Sym^{m-1}E\to \Sym^m E \to \Sym^m F\to 0.
$$
If $k$ has characteristic 0, and we have a short exact sequence of vector bundles
$$
0\to E\to F\to L\to 0
$$
with $L$ again a line bundle, then for any $m>0$ we have a short exact sequence
$$
0\to \Sym^m E \to \Sym^{m}F \to \Sym^{m-1} F \otimes L\to 0.
$$
\end{lem}

For the first fact see \cite[Proposition~A2.2]{eisenbud}; the second follows dualizing $0\to E\to F\to L\to 0$, applying the first fact, and then dualizing again, and using the identifications $(\Sym E^\vee)^\vee \cong \Sym^m E$, which holds only in characteristic $0$.

\section{$D$-simplicity of section rings and bigness of the tangent bundle}
\label{tangentbundle}

In this section, we recall \cite[Theorem~1.2]{bigness}:

\begin{thm}
\label{hsi}
Let $X$ be a smooth complex projective variety of dimension $\geq 2$, let $L$ be an ample line bundle on $X$, and let $R=S(X,L):=\bigoplus H^0(X,\O_X(mL))$ be the section ring of $X$ with respect to $L$.
If $R$ has a differential operator of negative degree (e.g., if $R$ is $D_R$-simple) then $T_X$ is big.
\end{thm}

We recall the proof from \cite{bigness} for the reader's convenience:

\begin{proof}
We first recall the connection between the tangent bundle of a smooth variety and the differential operators on its section ring from \cite{nakaiconj}.

Let $X$ be a smooth projective variety and $L$ an ample line bundle. We will assume for simplicity here that $R=S(X,L)$ is generated in degree 1. \cite{bigness} and \cite{nakaiconj} treat the general case; we note for our purposes here that we can also replace $R$ by a Veronese subring while preserving the existence of a differential operator of negative degree
(and thus reduce to the case here) and assume that $\Proj R$ is embedded in some $\P^n$ by $|L|$.

By Proposition~\ref{bigsimple}, $T_X$  is big if and only if for any $e<0$ there exists $m>0$ such that $H^0(X,\Sym^mT_X \otimes L^e)\neq 0$.
We claim  the vanishing $H^0(X,\Sym^mT_X \otimes L^e)=0$ for all $e<0$ and all $m$ implies on the other hand that $D^m_e:=(D^m_R)_e=0$ for all $e<0$ and all $m$, i.e., that $R$ has no differential operators of negative degree. This then shows that $R$ cannot be $D$-simple as the homogeneous maximal ideal will be a proper $D_R$-submodule.

Write $D^m$ for the differential operators on $R$ of order $\leq m$; write $D^m_l\subset D^m$ for the homogeneous differential operators of degree $l$.
Let $\widehat X=\Spec R$ be the affine cone over $X$, which embeds naturally in $\A^{n+1}=\Spec k[x_0,\dots,x_{n}]$, and let $U=\widehat X -\set{\mathfrak m}$ be the punctured cone, so
$\pi: U\to X$ is an $\A^1$-bundle.

  Let 
$$
I=
\sum_{i=0}^{n+1}  x_i \frac{\d}{\d x_i} \in D^1_0
$$
be the Euler operator on $R$ induced from that on 
$k[x_0,\dots,x_{n}]$.
Write $\Diff^m $ for the sheaf of differential operators of order $\leq m$ on $U$. Note that by reflexivity of $D^m$ we have that
$$
D^m = H^0(U,\Diff^m).
$$
Thus
$I$ gives a global section of $\Diff^1$, and let $\Diff^m_e\subset \Diff^m$ be the  subsheaf of differential operators $\delta$ with $[I,\delta]= e \delta$. The global sections of $\Diff^m_e$ are exactly those homogeneous differential operators~$\delta$ on $R$ such that for any homogeneous polynomial $f$ we have
$$
\deg \delta(f)-\deg f =e.
$$

For any $m,e$,
write $\Delta_e^m = \pi_*(\Diff^m_e)$;  note then that $$H^0(X,\Delta_e^m)=H^0(U,\Diff^m_e)=D^m_e.$$ One can then check:
\begin{enumerate}
\item $\Delta^m_e=\Delta^m_0\otimes L^e$ for any $m\geq 1$ and any $e\in \Z$.
\item Let $\sigma _1 =\Delta^1_0$ and let $\sigma_m =\Delta_0^m /\Delta_0^{m-1}$ for $m\geq 2$. Then 
$\sigma_m=\Sym^m \sigma_1$ and
we have a short exact sequence
\begin{equation}
\label{euler}
0\to \O_X \to \sigma_1 \to T_X\to 0,
\end{equation}
and thus by Lemma~\ref{sessym}
short exact sequences
\begin{equation}
\label{symeuler}
0\to \sigma_{m-1} \to \sigma_m \to \Sym^m T_X\to 0.
\end{equation}
\end{enumerate}

Now, let $e<0$.
Twisting \eqref{euler} by $L^e$ and taking global sections we get
$$
0\to H^0(X,L^e)\to H^0(X,\sigma_1\otimes L^e)\to H^0(X,T_X\otimes L^e)\to H^1(X,L^{e}).
$$
By Kodaira vanishing, $H^1(X,L^{e})=0$, while clearly $H^0(X,L^e)=0$.
Thus, if $H^0(X,T_X\otimes L^e)=0$, then 
$H^0(X,\sigma_1\otimes L^e)=0$. Moreover, since by definition $\sigma_1\otimes L^e=\pi_*(\Diff ^1_e)$, we have that 
$$
0=H^0(X,\sigma_1\otimes L^e)=H^0(U,\Diff^1_e)=D^1_e
$$
for $e<0$. 
Thus, if $H^0(X,T_X\otimes L^e)=0$ for $e<0$, then $R$ has no derivations of negative degree.

Now, we handle the higher order differential operators by induction on $m$.
Again, let $e<0$.
Twisting \eqref{symeuler} by $L^e$ and 
taking global sections we get
$$
0\to H^0(X,\sigma_{m-1}\otimes L^e) \to H^0(X,\sigma_m\otimes L^e) \to H^0(X,\Sym^m T_X\otimes L^e)
$$
Vanishing of $H^0(X,\Sym^m T_X\otimes L^e)$ for all $m$ and all $e<0$ will imply that $H^0(X,\sigma_m\otimes L^e)=H^0(X,\sigma_{m-1}\otimes L^e)$ for all $m$ and all $e<0$, and we've seen already that $H^0(X,\sigma_1\otimes L^e)=0$ for $e<0$, and thus we obtain $H^0(X,\sigma_m\otimes L^e)=0$ for all $m$ and all $e<0$.

By definition we have
$$
0\to \Delta^{m-1}_e\to \Delta^{m}_e \to \sigma_m \otimes L^e\to 0 
$$
and thus
$$
0\to \underbrace{H^0(X,\Delta^{m-1}_e)}_{D^{m-1}_e} \to \underbrace{H^0(X,\Delta^{m}_e)}_{D^{m}_e} \to H^0(X,\sigma_m \otimes L^e).
$$
Since the rightmost term vanishes,
we know that $D^m_e=D^{m-1}_e$ for all $e<0$, and we know already that $D^1_e=0$ for $e<0$, and thus the result is shown.
\end{proof}


\section{The tangent bundle of degree-3 del Pezzo surfaces}
\label{delpezzos}

In this section, we will treat the case of del Pezzo surfaces of degree 3, and show that their tangent bundles are not big.  In the next section, we will treat those of degree 4.
By \cite[Corollary~1.3]{bigness}, toric del Pezzo surfaces (i.e., those of degree $\geq 6$) have big tangent bundles, while combining the results of this section and the next implies those of degree $\leq 4$ do not have big tangent bundles (see Corollary~\ref{classific}).

The del Pezzo surfaces of degrees 3, because they embed as surfaces in $\P^3$.  If one attempts to use the resulting short exact sequences for their tangent bundles, however, one runs into difficulties; instead, the key is to study the cotangent bundles, via the following elementary fact:

\begin{lem}
\label{tangcotang}
For any smooth surface $Y$,
$T_Y\cong \Omega_Y(-K_Y)$.
\end{lem}

\begin{proof}
The nondegenerate pairing $\Omega_Y\times \Omega_Y \to \O_Y(K_Y)$ induces an isomorphism $(\Omega_Y)^\vee \cong \Omega_Y(-K_Y)$, and $(\Omega_Y)^\vee$ is simply $T_Y$.
\end{proof}

For the rest of this section we work over an arbitrary ground field of characteristic 0.

We begin by treating the case of del Pezzo surfaces of degree 3, i.e.,  smooth cubic surfaces. While our results for degree-4 del Pezzos  actually imply the results for those of degree-3, the argument is simpler in the degree-3 case, and the statement actually slightly stronger.
We will prove:

\begin{thm}\label{main}
Let $X$ be a del Pezzo surface of degree 3, i.e., a smooth cubic surface.
$T_X$ is not big; in fact, $H^0(X,\Sym^m T_X)=0$ for all $m$.
\end{thm}

This theorem immediately implies Theorem~\ref{notsuff}:
Set $R=S(X,\O_X(1))=\bigoplus_m H^0(X,\O_X(m))$.
By Theorem~\ref{hsi}, the theorem implies that $R$ has no differential operators of negative degree, and thus $R$ is not a simple $D_R$-module and $D_R$ is itself not a simple ring.
On the other hand,
 note that $X$ is Fano (since by adjunction $\omega_X=\O_X(-1)$). Thus $R$ has klt (thus also rational) singularities; since $X$ is a hypersurface $R$ is Gorenstein, and thus we've obtained the counterexample to Question~\ref{mainquest}  promised in Theorem~\ref{notsuff}. We note here that nothing in our results is specific to $\C[x,y,z,w]/(x^3+y^3+z^3+w^3)$, but applies just as readily to $\C[x,y,z,w]/(F)$ for any homogeneous cubic $F$ defining a smooth projective surface in $\P^3$.


\begin{rem}
Recent work of \cite{pseudoeff} has considered positivity properties of the tangent bundle of del Pezzo surfaces, and in particular of $T_X$. The notions of positivity they examine are analytic in nature, and in particular the definitions of ``big'' for vector bundles they consider is \emph{not} the same as the bigness of  $\O_{\P T_X}(1)$.  Thus, our result in Theorem~\ref{main} does not follow from  their results, and the methods we use here are much more algebraic in nature.
\end{rem}

\begin{lem}\label{cohP}
Let $n\geq 3$ and $m\geq 1$. Then:
\begin{enumerate}
\item $H^0(\P^n,\Sym^m \Omega_{\P^n}(e))=0$ for $e<m+1$.
\item $H^1(\P^n,\Sym^m \Omega_{\P^n}(e))=0$ for $e<m-1$.
\item $H^i(\P^n,\Sym^m \Omega_{\P^n}(e))=0$ for $1< i < n$ and any $e$.
\end{enumerate}
\end{lem}

\begin{proof}
The Euler sequence for $\Omega_{\P^n}$ is 
$$
0\to \Omega_{P^3}\to \O_{\P^n}(-1)^{\oplus n+1} \to \O_{\P^n}\to 0.
$$ 
Since all terms are locally free and the rightmost term has rank 1, by Lemma~\ref{sessym} we have a short exact sequence of symmetric powers
$$
0\to \Sym^m \Omega_{\P^n}\to \underbrace{\Sym^m(\O_{\P^n}(-1)^{\oplus n+1})}_{\bigoplus\O_{\P^n}(-m)} \to  \underbrace{\Sym^{m-1} (\O_{\P^n}(-1)^{\oplus n+1})}_{\bigoplus \O_{\P^n}(-m+1)}\to 0
$$
(the ranks are unimportant and we suppress them).
Twisting by  some $\O_{\P^n}(e)$ yields 
$$
0\to \Sym^m \Omega_{\P^n}(e)\to \bigoplus \O_{\P^n}(-m+e) \to \bigoplus \O_{\P^n}(-m+e+1)\to 0.
$$

Claim (1) of the lemma for $e<m$ follows (in any characteristic) by taking global sections, obtaining
$$
0\to H^0(\P^n,\Sym^m \Omega_{\P^n}(e))\to H^0\bigl(\P^n,\bigoplus \O_{P^3}(-m+e)\bigr),
$$
and noting that the right term vanishes for $-m+e<0$, but
for $e=m$ we must take a slightly different approach, for which a slight change in notation will be helpful:
Write $\P^n=\P(V)$ for a vector space $V$ of dimension $n$.
We twist the Euler sequence for the tangent bundle by $\O_{\P(V)}(1)$, obtaining 
$$
0\to 
\O_{\P(V)}(-1) \to
V^\vee
\otimes 
\O_{\P(V)}
\to 
T_{\P(V)}(-1)\to 0.
$$
Taking symmetric powers we obtain
$$
0\to 
\underbrace{\Sym^{m-1}(V^\vee\otimes \O_{\P(V)})(-1)}_{\Sym^{m-1}(V^\vee)\otimes \O_{\P(V)}(-1)}
\to
\underbrace{\Sym^m(V^\vee\otimes \O_{\P(V)})}_{\Sym^m(V^\vee)\otimes \O_{\P(V)}} \to
\Sym^m( T_{\P(V)}(-1))\to 0.
$$
Dualizing this sequence we have
$$
0\to 
\bigl(\Sym^m(T_{\P(V)}(1))\bigr)^\vee\to 
\Sym^m(V^\vee)^\vee \otimes\O_{\P(V)}\to
\Sym^{m-1}(V^\vee)^\vee\otimes \O_{\P(V)}(1)\to
0.
$$

Now, \emph{using that we're in characteristic 0}, we know that 
$$
\bigl(\Sym^m(T_{\P(V)}(1))\bigr)^\vee\cong \Sym^m(\Omega_{\P(V)}(1))=\Sym^m\Omega_{\P(V)}(m).
$$
Thus, to see that $H^0(\P^n,\Sym^m\Omega_{\P(V)}(m))=0$,
it suffices to show that the map
\begin{equation}\label{can}
H^0(\P^n,\Sym^m(V^\vee)^\vee \otimes\O_{\P(V)})\to
H^0\bigl(\P^n,\Sym^{m-1}(V^\vee)^\vee\otimes \O_{\P(V)}(1)\bigr)
\end{equation}
is injective.
But this is just the canonical map of vector spaces
$$
\Sym^m(V^\vee)^\vee \to \Sym^{m-1}(V^\vee)^\vee\otimes V.
$$
which is dual to the canonical multiplication
$$
\Sym^{m-1}(V^\vee)\otimes V^\vee \to  \Sym^m(V^\vee),
$$
which is obviously surjective, and thus \eqref{can} is injective and the $e=m$ case of claim (1) is shown.


For claim (2), the relevant terms of the long exact sequence are
$$
H^0\bigl(\P^n,\bigoplus  \O_{\P^n}(-m+e+1)\bigr)\to 
H^1(\P^n,\Sym^m \Omega_{\P^n}(e))\to 
H^1\bigl(\P^n,\bigoplus \O_{\P^n}(-m+e)\bigr).
$$
The right term vanishes always, while the left term is zero if $-m+e+1<0$.

Finally, claim (3) follows by examining the terms
$$
H^{i-1}\bigl(\P^n,\bigoplus  \O_{\P^n}(-m+e+1)\bigr)\to 
H^i(\P^n,\Sym^m \Omega_{\P^n}(e))\to 
H^i\bigl(\P^n,\bigoplus \O_{\P^n}(-m+e)\bigr) 
$$
and noting that the first and last terms vanish for any $e$ and $1< i <n$.
\end{proof}

\begin{lem}
\label{rest}
Let $m\geq 1$. Then:
\begin{enumerate}
\item 
$H^0(X,\Sym^m \Omega_{\P^3}\res X(m))=0$.
\item
$H^1(X,\Sym^{m}\Omega_{\P^3}\res X(m-3)) = 0$.
\end{enumerate}
\end{lem}

\begin{proof}
We start with (1):
Twisting the short exact sequence $0\to\O_{\P^3}(-3)\to\O_{\P^3}\to \O_X\to 0$ by $\Sym^m\Omega_{\P^3}(m)$, we have
$$
0\to 
\Sym^m\Omega_{\P^3}(m-3)\to
\Sym^m\Omega_{\P^3}(m)\to
\Sym^m\Omega_{\P^3}\res X(m)\to
0.
$$
Taking the long exact sequence in cohomology we get 
\begin{equation}\label{ses}
H^0(X,\Sym^m\Omega_{\P^3}(m))\to
H^0(X,\Sym^m\Omega_{\P^3}\res X(m))\to
H^1(X,\Sym^m\Omega_{\P^3}(m-3)).
\end{equation}
By Lemma~\ref{cohP} the first and last terms vanish, and thus
$H^0(X,\Sym^m\Omega_{\P^3}\res X(m))=0$, as desired.

For (2), we twist
$0\to\O_{\P^3}(-3)\to\O_{\P^3}\to \O_X\to 0$ by $\Sym^m\Omega_{\P^3}(m-3)$ and take the long exact sequence, yielding 
relevant terms
$$
H^1(\P^3,\Sym^m\Omega_{\P^3}(m-3))\to
H^1(X,\Sym^m\Omega_{\P^3}\res X(m-3))\to
H^2(\P^3,\Sym^m\Omega_{\P^3}(m-6)).
$$
Again, Lemma~\ref{cohP} says that the outer terms vanish and thus
$H^1(X,\Sym^m\Omega_{\P^3}\res X(m-3))=0$.
\end{proof}

\begin{proof}[Proof of Theorem~\ref{main}]
By \ref{tangcotang}, we have that 
$$\Sym^m(T_X)=
\Sym^m(\Omega_X(1))=\Sym^m(\Omega_X)(m),$$
and so it suffices to show that
$$
H^0(X,\Sym^m(\Omega_X)(m))=0
$$
for all $m$.

We have a presentation
$$
0\to \O_X(-3)\to \Omega_{\P^3}\res X \to \Omega_X\to 0
$$
for $\Omega_X$;
taking symmetric powers and twisting by $\O_X(m)$, we have
$$
0\to \Sym^{m-1}(\Omega_{\P^3}\res X)(m-3)\to \Sym^m \Omega_{\P^3}\res X(m) \to \Sym^m \Omega_X (m)\to 0.
$$

Taking the long exact sequence in cohomology we have
$$
H^0(X,\Sym^m \Omega_{\P^3}\res X(m ))\to
H^0(X,\Sym^m\Omega_X(m))\to 
H^1(X,\Sym^{m-1}\Omega_{\P^3}\res X(m-3)).
$$
But Lemma~\ref{rest} implies immediately that the outer terms vanish, and thus so does the middle term, proving the theorem.
\end{proof}

\section{An alternate proof, and the case of degree-4 del Pezzos}
\label{new}

As mentioned in the introduction, 
the author realized after the initial draft of our results was made public that Theorem~\ref{main} follows immediately from work of \cite{BO}. In this section, we'll briefly discuss this, and then
use related work of \cite{OL} to treat the case of del Pezzos of degree 4.

We recall first:
\begin{thm}[{\cite[Theorem~B]{BO}}]
If $X\subset \P^n$ is a smooth hypersurface, then for any $m>1$,
$$
H^0(X,\Sym^m \Omega_X (m))\neq 0
$$
if and only if $X$ is a hyperquadric in $\P^n$.
\end{thm}

Recall that if $X$ is a smooth cubic surface in $\P^3$, then we have an isomorphism $T_X\cong \Omega_X(1)$, and thus
applying the theorem we have immediately that
$$
H^0(X,\Sym^m T_X)= 
H^0(X,\Sym^m \Omega_X (m))= 0,
$$
thus recovering Theorem~\ref{main}.

We note that the proof of \cite[Theorem~B]{BO} involves a detailed study of the tangent map and the tangent 2-trisecant variety of the embedding $X\subset\P^n$, and thus the proof we gave in the preceding section is significantly more elementary, although correspondingly less general.

We now turn to the proof of the following theorem, using results of \cite{OL}:
\begin{thm}
Let $X$ be a del Pezzo surface of degree 4.
$T_X$ is not big.
\end{thm}

\begin{proof}
Let $X\subset \P^4$ is the anticanonical embedding of $X$ as a complete intersection of two quadrics in $\P^4$, say $Q_1,Q_2$. Let $f:\bigcup_{x\in X} T_x X\to \P^4$ be the tangent map of $X$, which associates to a point in a tangent plane to $x\in X$ the corresponding point of $\P^4$.
Combining Corollaries 2.1 and 3.1 of \cite{OL}, we have that if:
\begin{equation}
\label{star}
\tag{$*$}
 \text{$f$ is surjective with connected fibers}
\end{equation}
then there's a graded isomorphism
$ \bigoplus H^0(X,\Sym^m \Omega_X(m))=\Sym^{\bullet} H^0(X,\I_X(2))=\C[Q_1,Q_2] $
(where $\deg Q_1=\deg Q_2=2$). Assuming (1) and (2), then, we have that $H^0(X,\Sym^{2m}\Omega_X(2m))$ has as basis the set of degree-$m$ monomials in the $Q_i$, and thus grows like $m$ rather than $m^3$, and thus $\Omega_X(m)$ is not big.
Since we already
 know that 
$ T_X = \Omega_X(m)$ (using that $X$ is a surface embedded by its anticanonical divisor), this implies that $T_X$ is not big.

So, all that remains is to show that \eqref{star} holds. 
First, note that $\bigcup_{x\in X} T_x X $ has dimension~4, so to obtain surjectivity of the tangent map $f$ it suffices to check that it's generically finite. Since the tangent map is injective on each tangent plane $T_xX$, it suffices to check that a general point of $\P^4$ lies on only finitely many tangent planes to $X$.
This follows immediately, however, from the fact that the Gauss map $\gamma:X\to \Gr(2,\P^4)$ associating to a point $x\in X$ the tangent hyperplane $T_x X \in \Gr(2,\P^4)$  is not just generically finite, but birational (see  \cite[Corollary~2.8]{Zak}). 

Note that this implies that $f$ itself is generically injective and dominant, and thus $f$ is in fact birational.
This immediately gives connectivity of the fibers $f\inv(p)$ for $p\in \P^4$:
Since the tangent map $f$ is a birational morphism onto the smooth variety $\P^4$, Zariski's main theorem implies that $f$ has connected fibers, and thus the proof is complete.
\end{proof}

\begin{cor}
\label{classific}
Let $X_i$ be a del Pezzo surface of degree $i$. Then $T_{X_i}$ is not big for $i< 5$.
\end{cor}

\begin{proof}
First consider $X_4$, which is the blowup of $\P^2$ at five general points.
$-K_{X_4}$ embeds $X_4$ as the intersection of two smooth quadrics in $\P^4$, which we've just seen does not have big tangent bundle.
If $i<4$, we can view $X_i$ as the blowup of $X_4$ at $i-4$ general points, say $\mu: X_i\to X_4$. We have an injection
$$
T_{X_i}\hookrightarrow \mu^* T_{X_4};
$$
taking the $m$-th symmetric power yields a morphism
$$
\Sym^m T_{X_i}\hookrightarrow \Sym^m \mu^* T_{X_4}=\mu^* \Sym ^m T_{X_4}.
$$
This must be an injection, since 
$
T_{X_i}\hookrightarrow \mu^* T_{X_4}
$
is generically an isomorphism, so $\Sym^m T_{X_i}\to \Sym^m \mu^* T_{X_4}$ is generically an isomorphism well and thus has torsion kernel, but $\Sym^m T_{X_i}$ is locally free and thus cannot have a torsion subsheaf.
Taking global sections and noting that $H^0(X_i,\mu^* \Sym^m T_{X_4})=H^0(X_4,\Sym^m T_{X_4})$ since $\mu_* \O_{X_i}=\O_{X_4}$, we thus have a containment
$$
H^0(X_i,\Sym^m T_{X_i})\to H^0(X_4,\Sym^m  T_{X_4}),
$$
and thus
$H^0(\Sym^m T_{X_i})$ cannot grow like $m^3$.
\end{proof}

\begin{rem}
In particular, as mentioned at the beginning of this section, once we know that $T_{X_4}$ is not big, $T_{X_3}$ cannot be big either; however, our result above actually shows that $H^0(X_3,\Sym^m T_{X_3})=0$ for all $m$, which does not follow from our treatment of $T_{X_4}$, as we saw above  that $H^0(X_4,\Sym^2 T_{X_4})$ is $2$-dimensional.
\end{rem}

\section{A partial converse}
\label{conv}

The preceding section showed that given a Fano variety $X$ and an ample line bundle $L$, the section ring $S(X,L)$ may not be $D$-simple, even though it has only a Gorenstein rational singularity.
Even though this is not true, however, one can give a partial converse:

\begin{prop}
\label{uniruled}
Let $X$ be a smooth complex projective variety.
If $T_X$ is big then $X$ is uniruled.
\end{prop}

\begin{proof}

First, we recall the following theorem of Miyaoka:

\begin{thm}[{{\cite[Corollary~8.6]{Miyaoka}}}]
 If a smooth complex projective variety $X$ is not uniruled then $\Omega_X$ is generically nef, i.e., $\Omega_X\res C$ is nef for a general complete intersection curve $C\subset X$.
\end{thm}

Now, say $X$ is not uniruled but $T_X$ is big. Take $L$ to be an ample line bundle on $X$ and consider a nonzero global section
$s\in H^0( X,\Sym^m T_X \otimes L\inv)$. Choosing a general complete intersection curve $C\subset X$, which by generality will not lie in the zero locus of $s$, we obtain a nonzero global section $s\res C\in  H^0(C,\Sym^m T_X\res C \otimes L\res C\inv)$.
We can view this nonzero global section equivalently as an injection
$$
\O_C\hookrightarrow  \Sym^m T_X\res C\otimes L\res C\inv,
$$
or as an injection
$$
L\res C\hookrightarrow  \Sym^m T_X \res C.
$$
Moreover, we note that
$(\Sym^m T_X\res C)^\vee=\Sym^m \Omega_X \res C$, and that since $\Omega_X \res C$ is nef so is $\Sym^m \Omega_X \res C$ (by \cite[Theorem~6.2.12(iii)]{LazarsfeldII}).

\begin{lem}
If $C$ is a smooth curve, $L$ a line bundle on $C$ of positive degree (thus ample), and $E$ a vector bundle on $C$ with $E^\vee$ nef, then there is no injection $L\hookrightarrow E$.
\end{lem}

\begin{proof}
Say we have $L\hookrightarrow E$. The cokernel $Q:=E/L$ may not be torsionfree, but we may consider 
the surjection
$$
E\to Q\to Q/\text{torsion}.
$$
Since $C$ is a smooth curve, $Q':=Q/\text{torsion}$ is locally free, and thus 
so is the kernel of 
the surjection
$$
E\to Q'.
$$
Call this locally free kernel $L'$; it's clear $L'$ is a line bundle containing $L$, and thus $\deg L'\geq \deg L>0$.
The short exact sequence of locally free sheaves 
$$
0\to L'\to E\to Q'\to 0
$$
dualizes to 
$$
0\to (Q')^\vee \to E^\vee \to (L')^\vee \to 0.
$$
However, $E^\vee$ was supposed to be nef, yet the quotient $(L')^\vee$ is not, and we thus have a contradiction, so there can be no injection $L\hookrightarrow E$.
\end{proof}

Applying this lemma with $E=\Sym^m T_X$ and $L=L\res C$, we obtain Theorem~\ref{uniruled}.
\end{proof}

The following theorem  recovers and extends \cite[Corollary~4.49]{quantifying}, which treated the Gorenstein case. 

\begin{thm}
\label{partnec}
Let $R$ be a normal $\Q$-Gorenstein graded $\C$-algebra, generated in degree 1, with an isolated singularity. If $R$ is $D$-simple, then $\Proj(R)$ is Fano, and thus $R$ has klt singularities, and thus rational singularities.
\end{thm}

\begin{proof}
Let $X=\Proj R$, with $L=\O_X(1)$ the corresponding ample line bundle. Thus, we have that $R=S(X,L)$ is the section ring of the smooth projective variety $X$ under the 
projectively normal 
embedding defined by $L$. Since $R$ is $\Q$-Gorenstein, we must have that $K_X\sim r \cdot L$ for $r\in \Q$.
$D$-simplicity of $R$ forces $T_X$ to be big.
Thus, applying Theorem~\ref{uniruled}, we have that $X$ must be uniruled.

Since $X$ is uniruled, we must have $H^0(X,mK_X)=0$ for all $m>0$.
But for $m$ sufficiently large and divisible we have $mK_X\sim a \cdot L$ for $a=mr\in \Z$, $|a|\gg0$. Thus $H^0(X,a L)=0$ for $a\gg0 $, and by ampleness of $L$ we must have that $a<0$, so $r<0$ and $-K_X$ is ample. Thus $X$ is Fano and embedded by a multiple of its canonical divisor, so $S(X,L)$ has klt singularities by \cite[Lemma~3.1]{kollar}.
\end{proof}

\begin{rem}
If in the above $R$ is moreover Gorenstein, then $R$ has rational singularities (see Remark~\ref{propsings}).
\end{rem}


\section{Relationship to differential operators in characteristic $p$}
\label{poschar}

As mentioned in Section~\ref{diffops}, part of the motivation for the conjectural relationship between klt singularities and $D$-simplicity is the equivalence of $D$-simplicity and $F$-regularity for $F$-pure varieties, and the analogy between $F$-regularity in characteristic $p$ and klt singularities in characteristic 0.
In this section, we give a brief discussion of these analogies.

\begin{rem}
By \cite[Theorem~5.1]{SchwedeSmith}, a smooth (or klt) Fano variety $X$ over $\C$ has globally $F$-regular type; this means that if one looks at various models $X_p$ of $X$ over finite fields $\F_p$, then $X_p$ is globally $F$-regular for almost all $p$; this in turn is equivalent to the section ring $S(X_p,L_p)$ being strongly $F$-regular for any ample line bundle $L_p$ on $X_p$. We avoid giving a formal definition here, but the following example is indicative of the general process, at least in the case where $X$ can be defined over the subring $\Z\subset \C$:

\begin{exa}
Let $X=\Proj \C[x,y,z,w]/(x^3+y^3+z^3+w^3)$ be a smooth cubic surface. Then for each prime $p$, we have $X_p=\Proj \F_p[x,y,z,w]/(x^3+y^3+z^3+w^3)$; then we have a natural choice of section ring $S(X_p,L_p)=\F_p[x,y,z,w]/(x^3+y^3+z^3+w^3)$.
 For $p\geq 5$, it's easy to check that this is strongly $F$-regular (for example, by using Fedder's criteria \cite{Fedder}).
Thus, $X$ has $F$-regular type.
\end{exa}
\end{rem}

\begin{rem}
For any $p\geq 5$,
since $\F_p[x,y,z,w]/(x^3+y^3+z^3+w^3)$ is strongly $F$-regular (and $F$-pure), it is $D$-simple, and in particular has differential operators of negative degree.
However, our above results showed that $\C[x,y,z,w]/(x^3+y^3+z^3+w^3)$ is \emph{not} $D$-simple, as it has no differential operators of negative degree.
That is, there is no hope in general to ``lift'' differential operators of negative degree from characteristic $p$ to characteristic 0. This offers another example of the phenomena discussed in \cite{smith}, where the ring $R_p=(\Z/p\Z)[x,y,z]/(x^3+y^3+z^3)$ is shown to have a degree-0 differential operator for $p\equiv 1 \mod 3$ (i.e., those $p$ such that $R_p$ is $F$-split) that does not arise as the image of a differential operator on $\Z[x,y,z]/(x^3+y^3+z^3)$. As is noted there, this shows that there are ``more'' differential operators in positive characteristic. The example of this note is further evidence for this heuristic: in positive characteristic $p\geq 5$,
$\F_p[x,y,z,w]/(x^3+y^3+z^3+w^3)$ has differential operators of arbitrarily negative degree, while $\C[x,y,z,w]/(x^3+y^3+z^3+w^3)$ has no differential operators of negative degree.
\end{rem}

\begin{rem}
Recent work of
\cite{quantifying} introduced an invariant $s(R)$ of a ring $R$, called the \emph{differential signature}. One always has $0\leq s(R)\leq 1$, and if $s(R)>0$ then $R$ is $D_R$-simple.
We won't recall the definition of this invariant here, but
want to note briefly that
our
 results give an example of the contrasting behavior of $s(R)$ in positive characteristic and characteristic 0:

Again, let 
$R_p = \F_p[x,y,z,w]/(x^3+y^3+z^3+w^3)$
and 
$R = \C[x,y,z,w]/(x^3+z^3+z^3+w^3)$. Since $R$ is not $D_R$-simple, the differential signature of $R$ (over $\C$) must be zero. On the other hand, $R_p$ is strongly $F$-regular for each $p$, so it has positive $F$-signature; moreover, one can calculate using \cite{diagF} the limit of the $F$-signatures as $p$ goes to infinity to be 1/8. By \cite[Lemma~5.15]{quantifying}, this bounds the limit of the differential signatures of $R_p$ (over $\Z$) away from 0.  
Thus, one cannot expect to calculate the differential signature in characteristic 0 as a limit of differential signatures in characteristic $p$ as $p\to \infty$.
For further discussion on this question, see \cite[Section~5.3]{quantifying}.
\end{rem}

\begin{rem}
Although the main result of this paper is that the characteristic-0 analogue of ``strongly $F$-regular implies $D$-simple and $F$-pure'' is false,
another interesting connection to characteristic $p$ arises from considering the potential characteristic-0 converse, that is, does a $D$-simple ring with log canonical singularities necessarily have klt singularities?

We note that this follows from the conjectural relation between $F$-purity and log canonical singularities, as follows:
Let $R$ be a $D$-simple $\Q$-Gorenstein essentially finite-type $\C$-algebra and assume that $R$ has log canonical singularities. 
One can choose a finite-type $\Z$-algebra $A$ and an essentially finite-type $A$-algebra $R_A$ such that $R_A\otimes_\A \C=R$, and consider the reductions of $R_A$ modulo the expansion of various maximal ideals of $A$. For simplicity, we assume that we can take $A$ to be $\Z$, although the general case proceeds in the same way.

Conjecturally (see, e.g., \cite[Conjecture~2.4]{conjec}), since $R$ is log canonical there is a dense (not necessarily open) subset of $\Z$ such that the reduction $R_p$ is $F$-pure (i.e., $R$ is of dense $F$-pure type) for $p$ in this subset. By \cite[Theorem~5.2.1]{smithvdB}, $D$-simplicity of $R$ descends to $D$-simplicity of $R_p$ for $p$ in an \emph{open} dense subset of $\Z$ as well. An open dense set and an arbitrary dense set intersect in a dense subset, and thus over a dense subset of $\Z$, $R_p$ is $F$-pure and $D$-simple, and thus strongly $F$-regular by \cite[Theorem~2.2]{smith}. Thus, $R$ is of (dense) strongly $F$-regular type. Theorem~3.3 of \cite{Fpuretype} then implies that $R$ is klt (and thus also has rational singularities).

It would be interesting to have a proof that $D$-simple plus log canonical implies klt that does not rely on reduction to positive characteristic.
\end{rem}

\section{Big tangent bundles in characteristic 0}
\label{zerochar}

In this section, we briefly review what is known about bigness of the tangent bundle for smooth complex projective varieties; throughout, $X$ will denote such a variety.

\begin{rem}

By \cite{wahl}, 
if 
$X$ is a smooth projective variety and $L$ an ample line bundle, then
$H^0(X, T_X\otimes L^{-1})\neq 0$
 forces $X$ to be be projective space $\P^n$, and additionally $L=\O_{\P^n}(1)$  (except if $n=1$ in which case $L$ might be $\O_{\P^1}(2)$).
That is, if $\dim X \geq 2$ and $R=S(X,L)$ has a derivation of negative degree, then $X$ must be the projective $n$-space, and $R$ just a polynomial~ring.

It appears that less is known about the potential nonvanishing of
global sections of 
the higher symmetric powers
$H^0(X,\Sym^m T_X\otimes L^e)$.
The following result is as much as is known to us:

\begin{thm}[{{\cite[Theorem~6.3]{ADK}}}]
Let $X$ be a smooth complex projective variety of Picard number 1 and $L$ an ample line bundle.
If $H^0(X,T_X^{\otimes m}\otimes L^{-m})\neq 0$, then either $X=\P^n$ and $L=\O_{\P^n}(1)$ or $Q$ is a quadric hypersurface and $L$ is the restriction of the hyperplane class from the ambient projective space.
\end{thm}

Since in characteristic 0 we can embed $\Sym^m T_X\hookrightarrow T_X^{\otimes m}$, this implies that if $X$ is as above, and not a projective space or a hyperquadric, then
$H^0(X,\Sym^m T_X\otimes L^{-m})=0$.
That is, for such $X$ we can rule out differential operators on $R=S(X,L)$ of order $m$ and degree $-m$.
We do not know if one can extend this theorem to varieties with higher Picard number. 
\end{rem}

We emphasize that results of the above form are very specific to characteristic 0; for example, \cite{wahl} gives the example of the ring $R=(\Z/2\Z)[x_0,x_1,x_2]/(x_0^2+x_1x_2)$; $\Proj R$ is a smooth quadric, but $\d/\d x_0$ is a differential operator on $R$ of order 1 and degree $-1$.

\begin{rem}
Bigness of the tangent bundle of a smooth projective variety is known in the following cases:
\begin{itemize}
\item projective spaces.
\item quadrics (of any dimension).
\item varieties $X$ admitting an ample line bundle $L$ such that the section ring $S(X,L)$ is a split summand of a polynomial ring, and thus in particular:
\item smooth toric varieties.
\item Grassmannians and (partial) flag varieties.
\item when $T_X$ is nef (conjecturally, by \cite{CP} this is equivalent to $X$ being rational homogeneous) and $\dim X\leq 3$.
\item products of the above varieties.
\end{itemize}

Note that while many of these are Fano varieties, not all are (e.g., many toric varieties).  However, by \cite{bigness}, if $T_X$ is big \emph{and nef} then $X$ is Fano.  Nefness of $T_X$ is quite a restrictive condition though, and as already mentioned, it is conjectured in \cite{CP} to be equivalent to $X$ being rational homogeneous.
\end{rem}

We note here that other positivity properties of $T_X$ are well-studied, and appear to be quite restrictive: Beyond the aforementioned conjecture on nefness, there is the celebrated result of Mori \cite{Mori} proving a conjecture of Hartshorne that if $T_X$ is ample then $X\cong \P^n$.
It is thus natural to ask about the following question:

\begin{quest}
What conditions on a smooth complex variety $X$, beyond uniruledness, are imposed by bigness of the tangent bundle $T_X$?
\end{quest}

\begin{rem}[recent work]
\label{otherpaper}
Since an initial draft of this paper was made public, there has been additional progress towards this question through work of \cite{HLS}. There, the authors prove the following results:

\begin{enumerate}
\item (\cite[Theorem~1.2]{HLS}) Let $X_i$ be a del Pezzo surface of degree $i$. Then $T_X$ is big if and only if $i\geq 5$.
\item (\cite[Theorem~1.4]{HLS}) Let $X$ be a hypersurface of degree $d$ in $\P^n$ for $n\geq 3$. Then $T_X$ is big if and only if $d=2$.
\end{enumerate}

We note that the missing case that (1) settles is that of the del Pezzo of degree 5, as those of degree 6 or higher are toric (and hence have big tangent bundle) and those degree 4 or lower are covered by the results here. However, their methods are able to treat all the non-toric del Pezzos in a uniform way, via the study of the dual variety of minimal rational tangents. The case of the del Pezzo of degree 5 illustrates the subtlety of the question of when a Fano variety has big tangent bundle: the del Pezzo of degree 5 is not toric, and in fact has finite automorphism group, and yet has big tangent bundle, while the del Pezzos of lower degree do not have big tangent bundle.
(2) sheds further light on our question above, and indicates that one may expect bigness of $T_X$ to be quite restrictive, as it implies that projective space and hyperquadrics are the only smooth hypersurfaces to have big tangent bundle.
\end{rem}

\bibliographystyle{alpha}
\bibliography{link}

\end{document}